\documentclass[12pt]{amsart}
 
\pdfoutput=1
\usepackage[margin=1in]{geometry}
\usepackage{amsmath,amsthm,amssymb,amsrefs,bbm,color,esint,esvect,float,graphicx,mathrsfs}
\usepackage[bookmarksnumbered, colorlinks, plainpages,linkcolor=blue,anchorcolor=blue,citecolor=blue,urlcolor=blue]{hyperref}
\usepackage{todonotes}

\DeclareMathOperator{\supp}{supp \,}

\usepackage{euscript,latexsym}
\usepackage{accents}

\usepackage{epstopdf}
\AppendGraphicsExtensions{.tif}

\newcommand{\R}{\mathbb{R}}

 \newcommand{\vertiii}[1]{{\left\vert\kern-0.25ex\left\vert\kern-0.25ex\left\vert #1 
    \right\vert\kern-0.25ex\right\vert\kern-0.25ex\right\vert}}

 \newtheorem{thm}{Theorem}[section]
 \newtheorem{lemma}[thm]{Lemma}
 
 \newtheorem{prop}[thm]{Proposition}
 \newtheorem{rem}[thm]{Remark}

 \numberwithin{equation}{section}

\theoremstyle{definition}

\begin{document}

\title[Sharp weighted weak-type bounds for commutators of singular integrals]{Sharp weighted weak-type bounds for commutators of singular integrals}

\author{Adam Mair}
\address{Adam Mair, Department of Mathematics and Statistics, Bucknell University, Lewisburg, PA 17837, USA}
\email{adam.mair@bucknell.edu}

\author{\'Oscar Andr\'es Ram\'irez}
\address{\'Oscar Andr\'es Ram\'irez, School of Mathematical and Statistical Sciences, Clemson University, Clemson, SC 29634, USA}
\email{oscarar@clemson.edu}

\author{Cody B. Stockdale}
\address{Cody B. Stockdale, School of Mathematical and Statistical Sciences, Clemson University, Clemson, SC 29634, USA}
\email{cbstock@clemson.edu}
\thanks{C. B. Stockdale is supported by the National Science Foundation, DMS grant No. 2555710}

\begin{abstract}
We establish sharp weak-type bounds for commutators of Calder\'on--Zygmund operators and $\text{BMO}(\mathbb{R}^n)$ symbols with respect to $A_p$ weights. We prove that the bound \looseness=-1
$$
    \|[b,T]\|_{L^p(w)\rightarrow L^{p,\infty}(w)}\lesssim [w]_{A_p}^{\max(2,p')}
$$
for $p \in (1,\infty)$ and $w \in A_p$  is sharp for general $T \in \text{CZO}(\mathbb{R}^n)$ and $b \in \text{BMO}(\mathbb{R}^n)$. In the case of the standard logarithmic symbol, we show that the estimate improves to  
$$
    \|[\log|\cdot|,T]\|_{L^p(w)\rightarrow L^{p,\infty}(w)}\lesssim [w]_{A_p}^{\max\big(2,\frac{1}{p}+\frac{p'}{p}\big)}
$$
and that this dependence is optimal for this symbol. Our upper bound for the logarithmic symbol follows from the pointwise control of the commutator by a sum of the Hardy--Littlewood maximal operator and the squares of the radial Hardy operator and its adjoint, while our lower bounds are consequences of a general result for singular integral operators.
\end{abstract}



\maketitle


\section{Introduction}\label{IntroductionSection}
\allowdisplaybreaks[4]

We study weighted bounds for commutators $[b,T]$ of Calder\'on--Zygmund operators (CZOs) $T$ with multiplication by $\text{BMO}(\mathbb{R}^n)$ symbols $b$. Recall that $[b,T]$ is given by 
$$
    [b,T]f = bT(f) - T(bf).
$$
Following the seminal work of Coifman, Rochberg, and Weiss in \cite{CRW1976}, the $L^p(\mathbb{R}^n)$-boundedness of $[b,T]$ is generally dictated by the membership of $b$ in $\text{BMO}(\mathbb{R}^n)$. This theory extends to weighted spaces with respect to Muckenhoupt $A_p$ weights. In particular, following the work of Chung in \cite{C2011}, Chung, Pereyra, and P\'erez established the following sharp weighted strong-type $(p,p)$ bound in \cite{CPP2012}: if $p \in (1,\infty)$, $T \in \text{CZO}(\mathbb{R}^n)$, and $b \in \text{BMO}(\mathbb{R}^n)$, then 
\begin{align}\label{eq:SharpStrong}
	\|[b,T]\|_{L^p(w)\rightarrow L^p(w)}\lesssim [w]_{A_p}^{2\max\big(1,\frac{p'}{p}\big)}
\end{align}
for all $w \in A_p$, where $p':=\frac{p}{p-1}$. Moreover, they showed that this bound is the best possible among all choices of such $b$ and $T$. They obtained the upper bound via the Cauchy integral trick, and proved optimality for the logarithmic symbol $b(x) = \log|x|$ and the Hilbert transform $H$ using standard examples involving power weights. This upper bound can also be achieved by sparse domination methods; see \cite{LOR2017}. In \cite{P1995}, P\'erez extended the theory to the $p=1$ endpoint by establishing the $L\log L$ weak-type behavior of commutators; see \cites{A2018, CP2011, CM2012, PP2001, P1997, OC2011,CR2020} and references therein for more on the weak-type properties of commutators. \looseness=-1

We here consider sharp weighted weak-type $(p,p)$ bounds of commutators for $p \in (1,\infty)$. As $L^p(w)$ embeds into $L^{p,\infty}(w)$, the strong-type bound \eqref{eq:SharpStrong} implies that the weak-type $(p,p)$ bound for general commutators holds with norm at most a constant multiple of $[w]_{A_p}^{2\max(1,\frac{p'}{p})}$ for any $w \in A_p$. In \cite{PR2018}, P\'erez and Rivera-R\'ios improved this dependence to  
\begin{align}\label{eq:GeneralUpper}
	\|[b,T]\|_{L^p(w)\rightarrow L^{p,\infty}(w)} \lesssim [w]_{A_p}^{\max(2,p')}
\end{align}
using the sparse domination of \cite{LOR2017} and iterating the known sharp bounds of sparse operators. They also obtained a lower bound for the commutator of the Hilbert transform with the particular symbol $b(x) = \log|x| \in \text{BMO}(\mathbb{R})$. Given $p \in (1,\infty)$, $b \in \text{BMO}(\mathbb{R}^n)$, and $T \in \text{CZO}(\mathbb{R}^n)$, let $\varphi_{b,T}\colon[1,\infty)\rightarrow[0,\infty)$ be the asymptotically smallest function such that 
$$
    \|[b,T]\|_{L^p(w)\rightarrow L^{p,\infty}(w)} \leq
\varphi_{b,T}([w]_{A_p})
$$
holds for any $w \in A_p$, namely
$$
    \varphi_{b,T}(t) := \sup_{\substack{w \in A_p\\ [w]_{A_p}\leq t}} \|[b,T]\|_{L^p(w)\rightarrow L^{p,\infty}(w)}. 
$$ 
It was proved in \cite{PR2018} that 
\begin{align}\label{eq:PR2018Lower}
	\varphi_{\log|\cdot|,H}(t) \gtrsim
t^{\max\big(2,\frac{p'}{p}\big)}
\end{align}
as a consequence of the unweighted behavior of $\|[\log|\cdot|,H]\|_{L^{p}(\mathbb{R})\rightarrow L^{p,\infty}(\mathbb{R})}$ as $p\rightarrow 1^+$ and $p\rightarrow \infty$, following the methods developed in \cite{LPR2015}. While \eqref{eq:GeneralUpper} and \eqref{eq:PR2018Lower} show that the dependence $[w]_{A_p}^2$ is generally sharp for the weighted weak-type $(p,p)$ bounds when $p\ge 2$, their results exhibit a gap between the upper and lower bounds for $1<p<2$. 

The purpose of this paper is to close this gap. We first improve the lower bound of \eqref{eq:PR2018Lower}.
\begin{thm}\label{thm:LogLower}
If $p \in (1,\infty)$, then  
$$
	\varphi_{\log|\cdot|,H}(t) \gtrsim t^{\max(2,\frac{1}{p}+\frac{p'}{p})}.
$$
\end{thm}
\noindent Theorem \ref{thm:LogLower} follows from explicit examples involving power weights. To obtain the exponent $\frac{1}{p}+\frac{p'}{p}$, let $\delta \in (0,1)$, $w(x) = |x|^{(1-\delta)(p-1)}$, and $f(x) = |x|^{\delta-1}\chi_{(0,1)}(x)$. One can easily verify that  \looseness=-1
$$
    [w]_{A_p} \approx \delta^{1-p} \quad\text{and}\quad \|f\|_{L^p(w)} = \delta^{-\frac{1}{p}},
$$
while a change of variables and elementary estimates give 
$$
    |[b,H]f(x)| = \frac{1}{\pi}\int_0^1 \frac{\log(\frac{x}{y})}{x-y}y^{\delta-1}\,dy\gtrsim x^{\delta-1}\delta^{-2}
$$
for $x \in (0,1)$. These properties imply that 
\begin{equation*}
    \|[b,H]f\|_{L^{p,\infty}(w)} \gtrsim \delta^{-2} \approx [w]_{A_p}^{\frac{1}{p}+\frac{p'}{p}}\|f\|_{L^p(w)}.
\end{equation*}
Take $\delta\rightarrow 0$. The lower bound with exponent $2$ can similarly be established using the following ingredients: $\delta \in (0,1)$, $w(x) = |x|^{\delta-1}$, and $f(x) = |x|^{\delta(1-p')}\chi_{(1,\infty)}(x)$.

While Theorem \ref{thm:LogLower} sharpens the lower bound in the special case $b(x)=\log|x|$, it is still strictly smaller than the general upper bound of \eqref{eq:GeneralUpper} when $1<p<2$. Our next result improves the upper bound \eqref{eq:GeneralUpper} for the logarithmic symbol, closing the gap in this case. \looseness=-1
\begin{thm}\label{thm:LogUpper}
If $p \in (1,\infty)$ and $T\in \text{CZO}(\mathbb{R}^n)$, then 
$$
	\|[\log|\cdot|,T]\|_{L^p(w)\rightarrow L^{p,\infty}(w)}\lesssim [w]_{A_p}^{\max\big(2,\frac{1}{p}+\frac{p'}{p}\big)}
$$
for all $w \in A_p$. 
\end{thm}
\noindent Theorem \ref{thm:LogUpper} relies on a pointwise bound by a sum of the Hardy--Littlewood maximal operator and the squares of the radial Hardy operator and its adjoint; see Proposition~\ref{prop:CommutatorHardyBound} below. \looseness=-1 

While Theorems \ref{thm:LogLower} and \ref{thm:LogUpper} show that the dependence $[w]_{A_p}^{\max(2,\frac{1}{p}+\frac{p'}{p})}$ is sharp in the case of the logarithmic symbol, they do not answer our question for general $\text{BMO}(\mathbb{R}^n)$ symbols. Our next result shows that the estimate in \eqref{eq:GeneralUpper} is sharp among all $\text{BMO}(\mathbb{R}^n)$ symbols.\looseness=-1
\begin{thm}\label{thm:GeneralLower}
There exists $b \in \textrm{\normalfont{BMO}}(\mathbb{R})$ such that for any $p \in (1,\infty)$, it holds that 
$$
	\varphi_{b,H}(t) \approx t^{\max(2,p')}.
$$ 
\end{thm}
\noindent The standard logarithmic symbol and power weights are insufficient to establish Theorem \ref{thm:GeneralLower}, as Theorem \ref{thm:LogUpper} shows. Instead, we employ a family of lacunary power weights and a symbol constructed from corresponding localized logarithms; see Section \ref{section:GeneralLoer} below. 

The lower bounds in Theorem \ref{thm:LogLower} and Theorem \ref{thm:GeneralLower} are consequences of the following convenient general result, which gives lower bounds for integral operators in terms of the kernel's behavior. Below, we say that $\|\cdot\|_X\colon L^0(\mathbb{R}^n) \rightarrow [0,\infty]$ is a monotone functional if $\|f\|_X \ge \|g\|_X$ whenever $|f| \ge |g|$ almost everywhere and if $\|af\|_X=a\|f\|_X$ whenever $a\geq 0$, where we adopt the convention $0\cdot\infty=0$. We write $\sigma := w^{1-p'}$ for $w \in A_p$. 
\begin{thm}
\label{thm:SepVar}
Let $T$ be given for compactly supported $f \in \bigcup_{q \in (1,\infty)}L^q(\mathbb{R}^n)$ by 
$$
    Tf(x)=\int_{\mathbb{R}^n} K(x,y)f(y)\, dy
$$
for almost every $x\notin \supp\, f$, where $K\colon \mathbb{R}^{2n}\setminus\{(x,x)\colon x \in \mathbb{R}^n\} \rightarrow \mathbb{C}$. If there exist $s,\phi,\psi \in L^0(\mathbb{R}^n)$ such that  $|s|=1$, $\phi$ and $\psi$ are nonnegative, $\text{supp}\,\phi \cap \text{supp}\,\psi$ is a null set, and \looseness=-1
\begin{align}\label{eq:KernelLower}
    s(y)K(x,y)\geq \phi(x)\psi(y)
\end{align}
for almost every $x \in \text{supp}\,\phi$ and $y \in \text{supp}\,\psi$, then
$$
        \lVert T\rVert_{L^p(w)\to X}\geq \rVert \phi\rVert_{X}\lVert \psi \rVert_{L^{p^\prime}(\sigma)}
$$
for any $p\in(1,\infty)$, any $w\in A_p$, and any monotone functional $\lVert \cdot\rVert_X\colon L^0(\mathbb{R}^n)\rightarrow [0,\infty]$.
\end{thm}
\noindent In \eqref{eq:KernelLower}, we implicitly assume that the left-hand side is real for $x \in \text{supp}\,\phi$ and $y \in \text{supp}\,\psi$. 

\begin{rem}\label{rem:1.5}
Theorem \ref{thm:SepVar} applies to CZOs, implying the sharpness of the strong-type and weak-type bounds in the estimates 
$$
    \|T\|_{L^p(w)\rightarrow L^p(w)} \lesssim [w]_{A_p}^{\max\big(1,\frac{p'}{p}\big)} \quad\text{and}\quad \|T\|_{L^p(w)\rightarrow L^{p,\infty}(w)}\lesssim [w]_{A_p}
$$
of \cites{H2012,HLMORSU2012}. Indeed, the lower bounds for the Hilbert transform can be established as follows:
 If $s=-1$, $\phi(x)=\chi_{(0,1)}(x)$, $\psi(y)=\frac{1}{\pi y}\chi_{(1,\infty)}(y)$, and $w(x)=|x|^{\alpha-1}$ for $\alpha\in(0,1)$, then
    \begin{equation*}
        \|\phi\|_{L^p(w)}=\|\phi\|_{L^{p,\infty}(w)}\approx [w]_{A_p}^\frac{1}{p} \quad\text{and}\quad \|\psi\|_{L^{p^\prime}(\sigma)}\approx [w]_{A_p}^{\frac{1}{p^\prime}};
    \end{equation*}
if $s=1$, $\phi(x)=\frac{1}{\pi x}\chi_{(1,\infty)}(x)$, $\psi(y)=\chi_{(0,1)}(y)$, and $w(x)=|x|^{(p-1)(1-\delta)}$ for $\delta\in(0,1)$, then
    \begin{equation*}
        \|\phi\|_{L^p(w)}\approx [w]_{A_p}^{\frac{p'-1}{p}} \quad\text{and}\quad \|\psi\|_{L^{p^\prime}(\sigma)}\approx [w]_{A_p}^{\frac{1}{p}}.
    \end{equation*}
Additionally, Theorem \ref{thm:SepVar} recovers the optimality of the constant in the strong-type bound for commutators of \eqref{eq:SharpStrong}. Note that if $b \in \text{BMO}(\mathbb{R}^n)$ and $T$ is a CZO with kernel $K$, then $[b,T]$ can be expressed as in Theorem \ref{thm:SepVar} with kernel $K_b(x,y) := (b(x)-b(y))K(x,y)$. The lower bound for $\|[\log|\cdot|,H]\|_{L^p(w)\rightarrow L^p(w)}$ follows from Theorem \ref{thm:SepVar} as follows: If $s=1$, $\phi(x)=\chi_{(0,1)}(x)$, $\psi(y)=\frac{\log|y|}{\pi y}\chi_{(1,\infty)}(y)$, and $w(x)=|x|^{\alpha-1}$ with $\alpha\in(0,1)$, then 
    \begin{equation*}
        \lVert \phi\rVert_{L^p(w)}\approx [w]_{A_p}^\frac{1}{p} \quad \text{and}\quad \lVert \psi\rVert_{L^{p^\prime}(\sigma)}\gtrsim [w]_{A_p}^{1+\frac{1}{p^\prime}};
    \end{equation*}
if $s=1$, $\phi(x)=\frac{1}{\pi x}\chi_{(1,\infty)}(x)$, $\psi(y)=-\log|y|\chi_{(0,1)}(y)$, and $w(x)=|x|^{(p-1)(1-\delta)}$ with $\delta\in(0,1)$, then
    \begin{equation*}
        \lVert \phi\rVert_{L^p(w)}\approx [w]_{A_p}^\frac{p'-1}{p} \quad \text{and}\quad \lVert \psi\rVert_{L^{p^\prime}(\sigma)}\approx [w]_{A_p}^{\frac{p'+1}{p}}.
    \end{equation*}
\end{rem}

The examples of Remark \ref{rem:1.5} are instances in which there exist \mbox{$s\in\{-1,1\}$} and functions $\phi$ and $\psi$, such that the separation of variables \eqref{eq:KernelLower}
holds for $x\in \supp\,\phi$ and $\Psi:= \supp\, \psi$. If this kernel lower bound holds, then
\begin{equation}
\label{eq:sKLowBound}
    s\int_\Psi K(x,y)f(y)\,dy\geq \phi(x)\int_{\mathbb{R}^n} f(y) \psi(y)\, dy,
\end{equation}
whenever $f\geq 0$. On the other hand, Hölder's inequality gives
\begin{equation*}
    \int_{\mathbb{R}^n} f\psi\,dy\leq \lVert f\rVert_{L^p(w)}\lVert \psi\rVert_{L^{p^\prime}(\sigma)}.
\end{equation*}
We choose $f\ge 0$ such that $f^pw = \psi^{p'}\sigma$ so that equality holds above -- Theorem~\ref{thm:SepVar} follows from this optimal choice of $f$. 

The paper is organized as follows. In Section \ref{sec:Preliminaries}, we collect notation and preliminaries regarding CZOs, $\text{BMO}(\mathbb{R}^n)$, $A_p$ weights, and strong/weak Lebesgue spaces. In Section \ref{section:LogUpper}, we show that commutators of CZOs and the logarithmic symbol are dominated by the maximal operator plus the squares of the radial Hardy operator and its adjoint, and we prove Theorem \ref{thm:LogUpper}. In Section \ref{section:thm:LogLower}, we prove the general result of Theorem \ref{thm:SepVar} and apply it to establish Theorems \ref{thm:LogLower} and \ref{thm:GeneralLower}.


\section{Preliminaries}\label{sec:Preliminaries}

For $A,B>0$, we write $A \lesssim B$ if there exists $C>0$ (possibly depending on the underlying parameters $n$, $p$, $b$, and $T$, but not on weights $w$) such that $A \leq CB$, and we write $A \approx B$ if $A \lesssim B \lesssim A$. We denote the space of measurable functions on $\mathbb{R}^n$ by $L^0(\mathbb{R}^n)$. 

We consider linear singular integral operators $T$ given by 
\begin{align}\label{eq:SIOIntegral}
    Tf(x)= \int_{\mathbb{R}^n}K(x,y)f(y)\,dy
\end{align}
for $f \in C_c^{\infty}(\mathbb{R}^n)$ and almost every $x\not\in\text{supp}\,f$, where $K\colon \mathbb{R}^n\times\mathbb{R}^n\setminus \{(x,x)\colon x\in\mathbb{R}^n\}\rightarrow \mathbb{C}$ satisfies \looseness=-1
\begin{align*}
|K(x,y)|\lesssim \frac{1}{|x-y|^n}
\end{align*}
whenever $x \neq y$ and there exists $\delta>0$ such that 
\begin{align*}
    |K(x,y)-K(x,y')|, \,|K(y,x)-K(y',x)|\lesssim \frac{|y-y'|^{\delta}}{|x-y|^{n+\delta}}
\end{align*}
whenever $|y-y'|\leq \frac{1}{2}|x-y|$. A singular integral operator $T$ is a Calder\'on--Zygmund operator (CZO) if it extends boundedly on $L^2(\mathbb{R}^n)$. We denote the collection of all CZOs by $\text{CZO}(\mathbb{R}^n)$. A fact that will be useful for us is that the integral representation for CZOs given in \eqref{eq:SIOIntegral} holds for any $f \in \bigcup_{p \in [1,\infty)}L^p(\mathbb{R}^n)$ and $x \not \in \text{supp}\,f$; see \cite{GrafakosModern}*{Proposition 4.2.3}. Recall that the Hilbert transform is a Calder\'on--Zygmund operator with integral formula
$$
    Hf(x) = \frac{1}{\pi}\,\text{p.v.}\int_{\mathbb{R}}\frac{f(y)}{x-y}\,dy.
$$

Let $\text{BMO}(\mathbb{R}^n)$ denote the John--Nirenberg space of bounded mean oscillation, consisting of all $b \in L_{\text{loc}}^1(\mathbb{R}^n)$ such that
$$
    \|b\|_{\text{BMO}(\mathbb{R}^n)} := \sup_Q \langle |b-\langle b\rangle_Q|\rangle_Q <\infty,
$$
where the supremum is taken over all cubes $Q \subseteq \mathbb{R}^n$ and $\langle f \rangle_Q := \frac{1}{|Q|}\int_Q f(x)\,dx$. While it is clear that $L^{\infty}(\mathbb{R}^n) \subseteq \text{BMO}(\mathbb{R}^n)$, the function $b(x) = \log|x|$ is a canonical unbounded element of $\text{BMO}(\mathbb{R}^n)$. A fundamental property of BMO functions is the John--Nirenberg theorem, which implies that $|b|^q\in L^1_{\textrm{loc}}(\mathbb{R}^n)$ for every $q\in(0,\infty)$; see \cite{GrafakosModern}*{Corollary 3.1.8}. 

A weight is a locally integrable function that is positive almost everywhere. Given a weight $w$ and $p\in [1,\infty)$, we define $L^p(w)$ and $L^{p,\infty}(w)$ to be spaces of all $f \in L^0(\mathbb{R}^n)$ such that 
$$
    \|f\|_{L^p(w)}:=\Big(\int_{\mathbb{R}^n} |f|^pw\,dx\Big)^{1/p} \quad\text{and}\quad \|f\|_{L^{p,\infty}(w)}:=\sup_{\lambda>0}\lambda w(\{x \in \mathbb{R}^n\colon |f(x)|>\lambda\})^{1/p}
$$
are finite, respectively, where $w(A):= \int_A w\,dx$. Note that $\|\cdot\|_{L^p(w)}$ and $\|\cdot\|_{L^{p,\infty}(w)}$ are monotone functionals (in fact, $\|\cdot\|_{L^{p,\infty}(w)}$ is a quasi-norm and $\|\cdot\|_{L^p(w)}$ is a norm). For an operator $T$ and a monotone functional $X$, we denote the infimum of all $C>0$ such that 
$$
    \|Tf\|_{X} \leq C \|f\|_{L^p(w)}
$$
for all $f \in L^p(w)$ by $\|T\|_{L^p(w)\rightarrow X}$, if one exists, and otherwise $\|T\|_{L^p(w)\rightarrow X}=\infty$. 

We say that a weight $w$ is in $A_p$ for $p \in (1,\infty)$ if 
$$
    [w]_{A_p} := \sup_Q \langle w\rangle_Q\langle w^{1-p'}\rangle_Q^{p-1} < \infty. 
$$
The $A_p$ condition is well known to be sufficient for the strong-type $(p,p)$ bounds for CZOs and generally necessary for the weak-type $(p,p)$ bounds; see \cites{CF1974, HMW1973}. Typical examples of $A_p$ weights include power weights -- recall that $w(x) = |x|^{n(\alpha-1)}=|x|^{n(p-1)(1-\delta)}$ is an $A_p$ weight if and only if $0< \alpha <p$ or, equivalently, if and only if $0<\delta<p'$, and, in this case, one has 
$$
    [w]_{A_p} \approx \frac{1}{\alpha}\frac{1}{\delta^{p-1}}.
$$
Recall also that $w \in A_p$ if and only if $\sigma:=w^{1-p'} \in A_{p'}$ and, in this case, $[w]_{A_p}^{p'-1} = [\sigma]_{A_{p'}}$. We will also utilize the reverse-H\"older property of $A_p$ weights, which states that for every $w\in A_p$, there exists $r >1$ such that for every cube $Q \subseteq \mathbb{R}^n$, 
\begin{equation*}
    \langle w^r\rangle_Q^{1/r} \lesssim \langle w\rangle_Q.
\end{equation*}

The Hardy--Littlewood maximal operator $M$ is given for $f \in L^1_{\text{loc}}(\mathbb{R}^n)$ and $x \in \mathbb{R}^n$ by
$$
    Mf(x) := \sup_{r>0}\,\langle |f|\rangle_{B(x,r)}.
$$
In \cite{M1972}, Muckenhoupt showed that the $A_p$ condition also characterizes the weighted bounds for $M$. We will use the following sharp bounds proved by Buckley in \cite{B1993}: if $p \in (1,\infty)$, then 
\begin{align}\label{HLSharpBounds}
\|M\|_{L^p(w) \rightarrow L^p(w) } \lesssim [w]_{A_p}^{\frac{1}{p-1}} \quad\text{and} \quad \|M\|_{L^p(w) \rightarrow L^{p,\infty}(w)} \lesssim [w]_{A_p}^{\frac{1}{p}}
\end{align}
for all $w \in A_p$.


\section{Upper bound for the logarithmic symbol}\label{section:LogUpper}

In this section, we establish the sharp weighted weak-type $(p,p)$ bound for the commutator of a CZO with the logarithmic symbol, Theorem \ref{thm:LogUpper}. We begin by introducing the radial Hardy operator, $P$.  Letting $\omega_n$ be the volume of the unit ball in $\R^n$, we define
\[ 
    Pf(x) := \frac{1}{\omega_n|x|^n} \int_{|y|<|x|}f(y)\,dy
\]
for $f \in L^1_{\text{loc}}(\mathbb{R}^n)$ and $x \neq 0$. 
Applying Tonelli's theorem, the $L^2(\mathbb{R}^n)$ adjoint is given by
\[ P^*f(x) := \frac{1}{\omega_n} \int_{|y|>|x|} \frac{f(y)}{|y|^n}\,dy. \]
These operators are radial analogs of the classical Hardy operator introduced in \cite{MR1544414} that have been studied further in many important works such as \cites{MR0425598, MR1239796}. More recently, Sukochev, Yang, Zanin, and Zhou established a distributional relationship between the commutators of CZOs with $\text{BMO}(\mathbb{R}^n)$ symbols and the sum of squares of the Hardy operator and its adjoint in \cite{MR5008123}. \looseness=-1

To prove Theorem \ref{thm:LogUpper}, we establish the following pointwise upper bound for the commutator of a CZO with $\log|\cdot|$ in terms of the Hardy--Littlewood maximal operator and the sum of squares of the Hardy operator and its adjoint. One can compare our following pointwise bound with the distributional estimate of \cite{MR5008123}. 

\begin{prop}\label{prop:CommutatorHardyBound}
If $b(x) = \log |x|$ and $T \in \text{CZO}(\mathbb{R}^n)$, then 
\begin{equation*}
    |[b,T]f(x)|\lesssim Mf(x)+P^2|f|(x)+(P^*)^2|f|(x)
\end{equation*}
for every $f \in C_c^{\infty}(\mathbb{R}^n)$ and almost every $x \in \mathbb{R}^n$.
\end{prop}
\begin{proof}
    We first note that an application of Tonelli's theorem and switching to polar coordinates show that if $f \in L^0(\mathbb{R}^n)$ is nonnegative, then
    \begin{equation*}
        \begin{split}
            P^2f(x) ={}& \frac{n}{\omega_n |x|^n} \int_{|y|<|x|} \log \left( \frac{|x|}{|y|} \right)f(y)\,dy,\\
            (P^*)^2f(x) ={}& \frac{n}{\omega_n} \int_{|y|>|x|} \log \left( \frac{|y|}{|x|}\right) \frac{f(y)}{|y|^n}\,dy.
        \end{split}
    \end{equation*}
    The size condition of the kernel gives us the inequality
    \begin{align}\label{LogSizeBound}
        |[b,T]f(x)| \lesssim \int_{\R^n} \frac{\left|\log\left(\frac{|x|}{|y|}\right)\right|}{|x-y|^n}|f(y)|\,dy
    \end{align}
    for $x \notin \text{supp}\, f $. We would like to break up our domain of integration into the following three regions in order to reach our desired upper bound:
    \[ \{|y| < |x|/2\}, \quad \{ |x|/2 \leq |y| \leq 2|x|\}, \quad \text{and}\quad\{|y| > 2|x|\}. \]
   However, our regions may contain points in the support of $f$ where \eqref{LogSizeBound} does not directly apply.  Fortunately, due to the local Lipschitz regularity of the symbol $\log|\cdot|$ and $L^2(\mathbb{R}^n)$ approximations, the bound \eqref{LogSizeBound} applies in all three regions, and we can proceed. 
    
    For the inner region, since $|x-y| \geq |x|/2$, we have
    \begin{align*}
        \int_{|y| < |x|/2} \frac{\left|\log\left(\frac{|x|}{|y|}\right)\right|}{|x-y|^n}|f(y)|\,dy \lesssim \frac{1}{|x|^n} \int_{|y| < |x|}  \log\left( \frac{|x|}{|y|}\right) |f(y)|\,dy\lesssim P^2(|f|)(x).
    \end{align*}
    For the outer region, we similarly use $|x-y| \geq |y|/2$ to get
    \begin{align*}
        \int_{|y| > 2|x|} \frac{\left|\log\left(\frac{|x|}{|y|}\right)\right|}{|x-y|^n}|f(y)|\,dy \lesssim \int_{|y|>|x|} \log\left( \frac{|y|}{|x|} \right) \frac{|f(y)|}{|y|^n}\,dy \lesssim (P^*)^2(|f|)(x).
    \end{align*}
    For the middle region, the Mean Value Theorem and the imposed lower inequality yield
    \[ \left|\log\left(\frac{|x|}{|y|}\right)\right| \leq \frac{||x| - |y||}{\min(|x|,|y|)} \leq \frac{2|x-y|}{|x|}. \]
    The integral over this region is then bounded using the previous estimate, the imposed upper inequality on the region, and a decomposition into dyadic annuli as follows:
    \begin{align*}
        \int_{|x|/2 \leq |y| \leq 2|x|} \frac{\left|\log\left(\frac{|x|}{|y|}\right)\right|}{|x-y|^n}|f(y)|\,dy \lesssim{}& \frac{1}{|x|} \int_{|x-y| < 3|x|} \frac{|f(y)|}{|x-y|^{n-1}}\,dy\\
        ={}& \sum_{k=0}^\infty \frac{1}{|x|} \int_{3|x|2^{-(k+1)}\leq |x-y| < 3|x|2^{-k}} \frac{|f(y)|}{|x-y|^{n-1}}\,dy\\
        \lesssim{}& \sum_{k=0}^\infty \frac{2^{(k+1)(n-1)}}{(3|x|)^n} \int_{|x-y| < 3|x|2^{-k}}|f(y)|\,dy\\
        \lesssim{}& \sum_{k=0}^\infty 2^{-k}\langle |f|\rangle_{B(x,3\cdot 2^{-k}|x|)}\\
        \lesssim{}& Mf(x).
    \end{align*}
    Combining the inequalities from the above regions completes the proof.
\end{proof}

\begin{proof}[Proof of Theorem \ref{thm:LogUpper}]
    Assume by density that $f \in C_c^{\infty}(\mathbb{R}^n)$. Since the ball $B(0,|x|)$ is contained in $B(x,2|x|)$, we have that
    \begin{equation}\label{Hardy_PW_Inequality}
         |Pf(x)|\leq P|f|(x) \lesssim Mf(x).
    \end{equation}
Using \eqref{Hardy_PW_Inequality} and the bounds of \eqref{HLSharpBounds}, we have
\begin{equation}\label{SquareHardy_WeakNorm}
    \|P^2f\|_{L^{p,\infty}(w)} \lesssim [w]_{A_p}^{\frac{1}{p}}\|Mf\|_{L^p(w)}
    \lesssim [w]_{A_p}^{\frac{1}{p}+\frac{1}{p-1}} \|f\|_{L^p(w)}.
\end{equation}
For the term involving the adjoint $P^*$, we use duality, the strong-type bound of \eqref{HLSharpBounds}, and the fact that the dual weight $\sigma = w^{-\frac{1}{p-1}}$ satisfies $[\sigma]_{A_{p'}} = [w]_{A_p}^{p'-1}$ to see 
$$
    \|P^*\|_{L^p(w) \rightarrow L^p(w)} = \|P\|_{L^{p'}(\sigma) \rightarrow L^{p'}(\sigma)}
    \lesssim [\sigma]_{A_{p'}}^\frac{1}{p'-1}
    = [w]_{A_p}.
$$
We obtain the weak-norm bound for the square of the adjoint from its strong-type bound: 
\begin{equation}\label{AdjointSquareHardy_WeakNorm}
    \|(P^*)^2f\|_{L^{p,\infty}(w)} \leq \|(P^*)^2f\|_{L^p(w)} \lesssim [w]_{A_p}^2\|f\|_{L^p(w)}.
\end{equation}
Finally, Proposition \ref{prop:CommutatorHardyBound}, \eqref{SquareHardy_WeakNorm}, \eqref{AdjointSquareHardy_WeakNorm}, and the fact that $[w]_{A_p} \geq 1$ give that
\begin{align*}
    \|[b,T]f\|_{L^{p,\infty}(w)} \lesssim{}& \big([w]_{A_p}^\frac{1}{p} + [w]_{A_p}^{\frac{1}{p}+\frac{1}{p-1}} + [w]_{A_p}^2\big)\|f\|_{L^p(w)}\\
    \lesssim{}& [w]_{A_p}^{\max\big(2, \frac{1}{p}+\frac{p'}{p}\big)}\|f\|_{L^p(w)}.
\end{align*}
This concludes the proof.
\end{proof}


\section{Proofs of the lower bounds
}\label{section:thm:LogLower}

\subsection{Proof of Theorem \ref{thm:SepVar}}

We begin this section by proving the general lower bound for singular integrals whose kernels satisfy \eqref{eq:KernelLower}, Theorem \ref{thm:SepVar}. 

\begin{proof}[Proof of Theorem \ref{thm:SepVar}]
Fix $N\in \mathbb{N}$ and set $\nu=\min(\psi,N)$. Let $f=\nu^{p^{\prime}-1}\sigma$ and note that 
$$
    f\nu=f^pw=\nu^{p^\prime}\sigma.
$$ 
Let $E\subseteq \mathbb{R}^n$ be a compact set and put $f_E=f\chi_E$, $\nu_E=\nu \chi_E$. Since $\nu_E\in L^\infty(\mathbb{R}^n)$ and $\sigma \in L^1_{\textrm{loc}}(\mathbb{R}^n)$, we have 
\begin{equation*}
  \lVert f_E\rVert^p_{L^p(w)}=
  \lVert \nu_E\rVert_{L^{p^\prime}(\sigma)}^{p^\prime}=\lVert \nu^{p^\prime}\sigma\rVert_{L^1(E)}<\infty.
\end{equation*}
By the equality case of H\"older's inequality, we have 
\begin{equation}
\label{eq:IntfEpsi}
    \int f_E\nu\,dx =\int_E f\nu\, dx = \Bigl(\int_E f^p w\,dx\Bigr)^{1/p}\Bigl(\int_E\nu^{p^\prime}\sigma\,dx\Bigr)^{1/{p^\prime}}= \lVert f_E\rVert_{L^p(w)}\rVert \nu_E\rVert_{L^{p^\prime}(\sigma)}.
\end{equation}

The function $f_E$ is in $L^q(\mathbb{R}^n)$ for some $q \in (1,\infty)$ since it has compact support and is the product of an $L^\infty(\mathbb{R}^n)$ function and an $A_{p^\prime}$ weight. Indeed, by the reverse-Hölder property, there exists $q\in(1,\infty)$ 
such that $\langle \sigma^q\rangle_Q^{1/q}\lesssim \langle \sigma\rangle_Q$ for all cubes $Q\subseteq \mathbb{R}^n$. If $Q\supseteq E$, then 
\begin{equation*}
    \lVert f_E\rVert_{L^q(\mathbb{R}^n)}\leq \lVert \sigma \rVert_{L^q(E)} \lVert  \nu^{p'-1}\rVert_{L^\infty(\mathbb{R}^n)}\lesssim \lVert \sigma \rVert_{L^q(Q)} \lesssim |Q|^{\frac{1}{q}}\langle \sigma\rangle_Q<\infty.
\end{equation*}
By hypotheses, almost every $x\in \supp \phi$ lies outside $\supp \psi \supseteq \textrm{supp}\,f_E$, thus 
\begin{equation*}
    T(sf_E)(x)=\int_{\mathbb{R}^n} s(y)K(x,y)f_E(y)\, dy\geq \phi(x) \int_{\mathbb{R}^n} f_E(y)\nu(y)\,dy
\end{equation*}
for almost every $x \in \text{supp}\,\phi$. In general, $|T(sf_E)|\geq M \phi$ almost everywhere, where $M:=\int_{\mathbb{R}^n} f_E\nu\,dy$. Combining this and equation~\eqref{eq:IntfEpsi}, we find that 
\begin{equation*}
    \lVert T(sf_E)\rVert_{X}\geq M \lVert \phi\rVert_X= \lVert sf_E\rVert_{L^p(w)}\rVert \nu_E\rVert_{L^{p^\prime}(\sigma)} \lVert \phi\rVert_X
\end{equation*}
for any monotone functional $\lVert \cdot\rVert_X$. The monotone convergence theorem yields the result.
\end{proof}


\subsection{Proof of Theorem \ref{thm:LogLower}}

We next prove the lower bound for the weak-type estimate for the commutator of $H$ with the logarithmic symbol, Theorem \ref{thm:LogLower}.

\begin{prop} 
\label{prop:thm1.1.1}
Let $p \in (1,\infty)$ and $w\colon \mathbb{R}\rightarrow [0,\infty)$ be given by $w(x)=|x|^{\alpha-1}=|x|^{(p-1)(1-\delta)}$ for $\alpha \in (0,p)$, equivalently, $\delta \in (0,p')$. If $\alpha\in(0,1)$, then  
    \begin{equation*}
        \lVert [\log|\cdot|,H]\rVert_{L^p(w)\to L^{p,\infty}(w)}\gtrsim [w]_{A_p}^2, 
    \end{equation*}
    and if $\delta\in(0,\frac{1}{1+\log 2}]$, then
    \begin{equation*}
        \lVert [\log|\cdot|,H]\rVert_{L^p(w)\to L^{p,\infty}(w)}\gtrsim [w]_{A_p}^{\frac{1}{p}+\frac{p^\prime}{p}}.
    \end{equation*}
\end{prop}
\begin{proof}
Let $b(x)=\log|x|$. If $q\in(1,\infty)$ and $f\in L^q(\mathbb{R})$ has compact support, then $b f\in L^1(\mathbb{R})$ as a consequence of the John--Nirenberg theorem and Hölder's inequality as $\|b f\|_{L^1(\mathbb{R})}\leq \|b\|_{L^{q'}(\supp f)}\|f\|_{L^q(\mathbb{R})}$. Thus $Hf$ and $H(bf)$ have integral representations, and
\begin{equation*}
    [b,H]f(x)=\bigl(bHf-H(bf)\bigr)(x)=\int_\mathbb{R}K(x,y)f(y)\, dy,
\end{equation*}
for almost every $x\notin \supp f$, where $K(x,y):= \frac{b(x)-b(y)}{\pi(x-y)}$. We have the following two lower bounds for the Hilbert kernel:
\begin{equation}
\label{eq:KHilLowBounds}
\begin{split}
    \frac{1}{\pi(x-y)}\geq \frac{1}{\pi x},\quad \quad x\in (1,\infty), \,\, y\in(0,1),\\
    -
    \frac{1}{\pi(x-y)}\geq \frac{1}{\pi y},\quad \quad x\in (0,1), \,\, y\in(1,\infty).
\end{split}
\end{equation}
If $x$ and $y$ are as in either of these two cases, then $b(x)$ and $b(y)$ have different signs and 
\begin{equation*}
    |K(x,y)| = \Bigl|\frac{b(x)-b(y)}{\pi(x-y)}\Bigr|\geq \Big|\frac{b(x)}{\pi(x-y)}\Big|.
\end{equation*}

The second case of \eqref{eq:KHilLowBounds} yields 
\begin{equation*}
    K(x,y)\geq |b(x)|\chi_{(0,1)}(x) \Bigl(\frac{1}{\pi y}\chi_{(1,\infty)}(y)\Bigr)
\end{equation*}
for $x\in (0,1)$ and $y\in(1,\infty)$. Applying Theorem~\ref{thm:SepVar}, we have 
\begin{equation*}
\begin{split}
    \lVert [b,H]\rVert_{L^p(w)\to L^{p,\infty}(w)}&\geq \Bigl\lVert \frac{1}{\pi |\cdot|}\chi_{(1,\infty)}\Bigr\rVert_{L^{p^\prime}(\sigma)}\lVert b\chi_{(0,1)}\rVert_{L^{p,\infty}(w)}.
\end{split}
\end{equation*}
Noting that $[w]_{A_p}\approx \alpha^{-1}$ for $\alpha \in (0,1)$ and that $\sigma(x)=|x|^{\delta-1}$, we have 
\begin{equation*}
    \Bigl\lVert \frac{1}{\pi |\cdot|}\chi_{(1,\infty)}\Bigr\rVert^{p'}_{L^{p^\prime}(\sigma)}=\frac{1}{\pi^{p^\prime}}\int^\infty_1 x^{-p^\prime} x^{\delta-1}\, dx\approx \frac{1}{p^\prime-\delta}=\frac{p-1}{\alpha}\approx [w]_{A_p}.
\end{equation*}
Since $|b|$ is decreasing on $(0,1)$, we have that  $|b(x)|\geq \frac{1}{\alpha}$ for $x\in(0,e^{-1/\alpha})$, and so
\begin{equation*}
\begin{split}
    \lVert b\chi_{(0,1)}\rVert^p_{L^{p,\infty}(w)}&\geq \frac{1}{\alpha^p} w\bigl((0,e^{-1/\alpha})\bigr)=\frac{e^{-1}}{\alpha^{p+1}}\approx [w]_{A_p}^{p+1}.
\end{split} 
\end{equation*}
In conclusion,
\begin{equation*}
    \lVert [b,H]\rVert_{L^p(w)\to L^{p,\infty}(w)}\gtrsim [w]^{\frac{1}{p^\prime}}_{A_p} [w]^{1+\frac{1}{p}}_{A_p}=[w]_{A_p}^2.
\end{equation*}

Theorem~\ref{thm:SepVar} and the first case of \eqref{eq:KHilLowBounds} similarly yield
\begin{equation*}
    \lVert [b,H]\rVert_{L^p(w)\to L^{p,\infty}(w)}\gtrsim\lVert \chi_{(0,1)}\rVert_{L^{p^\prime}(\sigma)}\Bigl\lVert \frac{b}{\pi |\cdot|}\chi_{(1,\infty)}\Bigr\rVert_{L^{p,\infty}(w)}.
\end{equation*}
 Noting that $[w]_{A_p}\approx \delta^{1-p}$ for $\delta \in (0,1)$, we have $\lVert \chi_{(0,1)}\rVert_{L^{p^\prime}(\sigma)}\approx \delta^{-\frac{1}{p^\prime}}\approx [w]_{A_p}^{\frac{1}{p}}$. Now, let $y=e^{1/\delta}$. If, in addition, $\delta\leq \frac{1}{1+\log 2}$, then $\frac{y}{2}\geq e$ and the function $h(t)=\frac{\log|t|}{t}$ is positive and decreases on $[\frac{y}{2},\infty)$, which means $h(t)\geq h(y)$ whenever $t\in (\frac{y}{2},y)$. Therefore, 
\begin{equation*}
\begin{split}
    \Bigl\lVert \frac{b}{\pi |\cdot|}\chi_{(1,\infty)}\Bigr\rVert^p_{L^{p,\infty}(w)}&\geq\frac{1}{\pi^p } h(y)^p w\bigl((y/2,y)\bigr)\gtrsim \frac{1}{(\delta y)^p}y^{\alpha}\\
    &=\frac{1}{\delta^p}y^{(1-p)\delta}=\frac{1}{\delta^p} e^{1-p}\approx [w]^{p'}_{A_p}.
\end{split}
\end{equation*}
In conclusion,
\begin{equation*}
    \lVert [b,H]\rVert_{L^p(w)\to L^{p,\infty}(w)}\gtrsim[w]_{A_p}^{\frac{1}{p}+\frac{p'}{p}}. \qedhere
\end{equation*}
\end{proof}

\begin{proof}[Proof of Theorem \ref{thm:LogLower}]
Let $w_\alpha(x)=|x|^{\alpha-1}$ for $\alpha\in(0,1)$. Take $M>0$ so that $ [w_\alpha]_{A_p}\leq \frac{M}{\alpha}$ for every $\alpha$. For $t>M$, fix $\alpha=\frac{M}{t}$, then $[w_\alpha]_{A_p}\approx \frac{1}{\alpha}\approx t$ and, by Proposition \ref{prop:thm1.1.1}, we have
\begin{equation*}
\begin{split}
    \varphi_{\log|\cdot|,H}(t)
    \geq \lVert [\log|\cdot|,H]\rVert_{L^p(w_\alpha)\to L^{p,\infty}(w_\alpha)}\gtrsim  t^2.
\end{split}
\end{equation*}
Similarly, let $w_\delta=|x|^{(p-1)(1-\delta)}$ for $\delta\in(0,\frac{1}{1+\log 2})$. Take $M>0$ so that \mbox{$ [w_\delta]_{A_p}\leq \frac{M}{\delta^{p-1}}$} for each $\delta$. For $t>M(1+\log 2)^{p-1}$, fix $\delta=(\frac{M}{t})^{\frac{1}{p-1}}$, then $[w_\delta]_{A_p}\approx t$ and, by Proposition \ref{prop:thm1.1.1}, we have
\begin{equation*}
    \varphi_{\log|\cdot|,H}(t)\geq \lVert [\log|\cdot|,H]\rVert_{L^p(w_\delta)\to L^{p,\infty}(w_\delta)}\gtrsim t^{\frac{1}{p}+\frac{p'}{p}}. \qedhere
\end{equation*}
\end{proof}

\subsection{Proof of Theorem~\ref{thm:GeneralLower}}\label{section:GeneralLoer}
We now prove the lower bound for the weak-type estimate of general commutators, Theorem \ref{thm:GeneralLower}. We first define an appropriate family of lacunary weights. This construction is related to the example from \cite{M2022}*{Section 3.1}, which was used to prove the sharpness of the weighted $L^2$ bound for strong-sparse operators. We next introduce a symbol $b \in \text{BMO}(\mathbb{R})$, adapted to this family of lacunary weights, for which the bound \eqref{eq:GeneralUpper} is saturated. We finish by proving Theorem \ref{thm:GeneralLower}.

\subsubsection{Construction of the lacunary weights}
For $p\in(1,\infty)$, let $N_p$ be the smallest natural number such that $N_p^{-\frac{1}{p-1}}<\frac{1}{8}$. For every integer $m\geq N_p$ and $k \in \mathbb{Z}$, let 
\begin{equation*}
    \mu_{p}:= m^{-1/(p-1)}, \quad r_k := 2^{-k},  \quad\text{and}\quad M_{p,k}:= ((1-\mu_p)r_k,(1+\mu_p)r_k).
\end{equation*}
Define $w_{p,m}$ on $[0,\infty)$ as follows:
\begin{equation}
\label{eq:wpmDef}
    w_{p,m}(x):=\begin{cases}
        r_k^{\frac{1}{m}-1}\Bigl(\frac{|x-r_k|}{\mu_p r_k}\Bigr)^{(p-1)(1-\mu_p)} & x\in M_{p,k},\\
        |x|^{\frac{1}{m}-1} & x\notin \bigcup\limits_{k\in\mathbb{Z}}M_{p,k},
    \end{cases}
\end{equation}
and extend $w_{p,m}$ evenly: $w_{p,m}(-x):=w_{p,m}(x)$ for all $x>0$.

\begin{lemma}
\label{lemLB2}
If $p \in (1,\infty)$ and $m\geq N_p$ is an integer, then $w_{p,m} \in A_p$ with \looseness=-1
$$
    [w_{p,m}]_{A_p}\approx m,
$$
where $w_{p,m}$ is as defined in \eqref{eq:wpmDef} and $N_p$ is the smallest natural number such that $N_p^{-\frac{1}{p-1}}<\frac{1}{8}$.
\end{lemma}
\begin{proof}
For simplicity, we use the following notation:
\begin{equation*}
    \varepsilon=\frac{1}{m},\quad \mu=\mu_{p,m}=m^{-\frac{1}{p-1}},\quad M_{k}=M_{p,k}, \quad w=w_{p,m},\quad \text{and}\quad \sigma= (w_{p,m})^{-\frac{1}{p-1}}.
\end{equation*}
If $a=(1-\varepsilon)/(p-1)$, then 
\begin{equation*}
    \sigma(x)=\begin{cases}
        r^{a}_k \Bigl(\frac{|x-r_k|}{\mu r_k}\Bigr)^{\mu-1}  & x\in M_k,\\
        |x|^a & x\notin \bigcup\limits_{k\in \mathbb{Z}}M_k,
    \end{cases}
\end{equation*}
for $x\geq 0$ and $\sigma$ inherits evenness from $w$.

\smallskip
\noindent
\textbf{Lower bound  \texorpdfstring{\boldmath$[w]_{A_p}\gtrsim m$}{}:} We prove that for $K\in\mathbb{Z}$, $R=r_K$, and $I=[0,R]$, one has 
\begin{equation*}
     \int_I w\,dx \gtrsim m R^\varepsilon \quad \text{and}\quad\int_I \sigma\,dx\gtrsim R^{a+1},
\end{equation*}
which implies that $[w]_{A_p}\gtrsim \langle w\rangle_{I}\langle \sigma\rangle^{p-1}_{I}\gtrsim m R^{\varepsilon+(a+1)(p-1)-p}=m$. Indeed, observe that $x^{\varepsilon-1}\leq ((1-\mu)r_k)^{\varepsilon-1}$ on $M_k$, so
\begin{equation}
\label{eq:4.7}
    \int_{M_k} x^{\varepsilon-1}\, dx\leq 2\mu(1-\mu)^{\varepsilon-1} r_k^{\varepsilon}\leq 3\mu r^\varepsilon_k, 
\end{equation}
where the last inequality comes from $\mu<\frac{1}{8}$. Summing a geometric series and using that $\frac{\varepsilon}{1-2^{-\varepsilon}}\leq 2$ for  $\varepsilon\leq 1$, we bound 
\begin{equation*}
\begin{split}
    \int_I w \,dx&\geq \int^R_0 x^{\varepsilon-1}\, dx- \sum_{k\geq K} \int_{M_k} x^{\varepsilon-1}\, dx\\
    &\geq mR^\varepsilon-3\mu \frac{(2^{-K})^\varepsilon}{1-2^{-\varepsilon}}\\
    &=m R^\varepsilon \Bigl(1-3\mu \frac{\varepsilon}{1-2^{-\varepsilon}}\Bigr)\\
    &\geq  mR^\varepsilon(1-6\mu).
\end{split}
\end{equation*}
This implies $\int_I w \,dx> \frac{1}{4}mR^\varepsilon$ for $\mu<1/8$. Finally, note that
\begin{equation}
\label{eq:4.8}
     \int_{M_k}\sigma \,dx =r^{a+1}_k\mu \int^1_{-1}|x|^{\mu-1}\, dx=2r^{a+1}_k.
\end{equation}
Similarly, $\int^{r_k}_{(1-\mu)r_k}\sigma=r_k^{a+1}$. 
Thus $M_K\cap (0,R)=((1-\mu)r_K,r_K)\subseteq I$ implies $\int_I \sigma \,dx\geq  R^{a+1}$.

\smallskip
\noindent
\textbf{Upper bound  \texorpdfstring{\boldmath$[w]_{A_p}\lesssim m$}{}:} We first show that for $R>0$ and $I=[0,R]$, one has 
\begin{equation*}
     \int_I w\,dx\lesssim m R^\varepsilon \quad\text{and}\quad \quad\int_I \sigma\,dx\lesssim R^{a+1}.
\end{equation*}
As in the previous part, the exponents of $R$ here imply $\langle w\rangle_I \langle \sigma\rangle^{p-1}_I\lesssim m$. See that between $M_k$ and $M_{k-1}$, we have 
\begin{equation}
\label{eq:4.9}
    \int^{(1-\mu)r_{k-1}}_{(1+\mu)r_k} w\,dx = r_k^\varepsilon \frac{(2-2\mu)^\varepsilon-(1+\mu)^\varepsilon}{\varepsilon}\leq r^\varepsilon_k\frac{2^\varepsilon-1}{\varepsilon}\leq r^\varepsilon_k,
\end{equation}
where we have used that $\frac{2^\varepsilon-1}{\varepsilon}\leq 1$ when $\varepsilon\leq 1$.
Take the smallest $K\in\mathbb{Z}$ such that $M_K$ intersects $I$. Then $\frac{1}{2}r_K<R$. Since $w(x)\lesssim x^{\varepsilon-1}$, using \eqref{eq:4.7} and \eqref{eq:4.9}, we get
\begin{equation*}
    \int_I w\,dx \leq \sum_{k\geq K} (3\mu +1)r^\varepsilon_k=(3\mu+1)\frac{r_K^{\varepsilon}}{1-2^{-\varepsilon}}\leq m(2R)^\varepsilon(3\mu+1)\frac{\varepsilon}{1-2^{-\varepsilon}}<6mR^\varepsilon. 
\end{equation*}
Similarly, between $M_k$ and $M_{k-1}$, we have 
\begin{equation*}
    \int^{(1-\mu)r_{k-1}}_{(1+\mu)r_k} \sigma \,dx \leq r^{a+1}_k \frac{2^{a+1}}{a+1}\lesssim_p r^{a+1}_k.
\end{equation*}
This bound and \eqref{eq:4.8} imply
\begin{equation*}
    \int_I \sigma\,dx \lesssim \sum_{k\geq K}r^{a+1}_k\leq \frac{r^{a+1}_K}{1-2^{-a-1}}\lesssim R^{a+1},
\end{equation*}
as desired.

If $I=(-r_1,r_2)$ with $r_1,r_2\geq 0$, one of them different from zero, evenness of $w$ and $\sigma$ imply that if $R=\max\{r_1,r_2\}$ and $J=[0,R]$, then 
\begin{equation*}
    \langle w\rangle_I \langle \sigma\rangle^{p-1}_{I}\leq \langle 2w\rangle_J \langle 2\sigma\rangle_J^{p-1}\lesssim m.
\end{equation*}
Thus, we shall restrict our analysis to the cases where $I$ does not intersect the origin. Without loss of generality, we take $I\subseteq(0,\infty)$.

Let $I=(R,R+L)$ with $R,L>0$. If $L\geq R/3$, then $R+L\leq 4L$ and again an estimate using the first case establishes the desired upper bound. Assuming $L<R/3$, we prove that $I$ meets at most one set $M_k$. Let $k$ be the largest integer such that $R<(1+\mu)r_k$. Then $\frac{1}{2} r_k<(1+\mu)r_{k+1}\leq R$, and so $I\cap M_{k+1}=\varnothing$. Also, we have that 
\begin{equation*}
    L<\frac{R}{3}\quad\text{and}\quad  \mu<\frac{1}{8}  \quad \implies \quad  R+L<\frac{7}{4}r_k<(1-\mu)r_{k-1},
\end{equation*}
so $I\cap M_{k-1}=\emptyset$. 

If $I\subseteq M_k$, then $w$ behaves in $I$ as a multiple of $|x|^{(p-1)(1-\mu)}$ around 0 in the following sense. Let $J=\{\frac{x-r_k}{\mu r_k}\colon x\in I\}\subseteq (-1,1)$, which is an interval with length $|J|=|I|/\mu r_k$, so
\begin{equation*}
    \langle w\rangle_I= \frac{1}{|I|}\int_I r^{\varepsilon-1}_k\Bigl(\frac{|x-r_k|}{\mu r_k}\Bigr)^{(p-1)(1-\mu)}\, dx =  \frac{\mu r^\varepsilon_k }{|I|}\int_{J} |x|^{(p-1)(1-\mu)}\, dx= r_k^{\varepsilon-1}\langle |\cdot|^{(p-1)(1-\mu)}\rangle_J.
\end{equation*}
The equivalent expression $\langle \sigma\rangle_I^{p-1}=r_k^{1-\varepsilon}\langle |\cdot|^{\mu-1}\rangle^{p-1}_J$ can be obtained similarly. Using the known $A_p$ characteristic for power weights,
\begin{equation*}
    \langle w\rangle_I\langle \sigma\rangle_I^{p-1}\leq [|\cdot|^{(p-1)(1-\mu)}]_{A_p}\lesssim\mu^{1-p}=m.
\end{equation*}

Now, assume $I$ is not contained in $M_k$, and let $I_1= I\cap M_k$ and $I_2= I\setminus I_1$. Since $L<R/3$ and $w(x)\lesssim |x|^{\varepsilon-1}$ (also the case for $x\in M_k$), we have
\begin{equation*}
    \langle w\rangle_I\lesssim \frac{1}{L}\frac{(R+L)^\varepsilon-R^\varepsilon}{\varepsilon}\lesssim R^{\varepsilon}\frac{\log(1+L/R)}{L}\leq R^{\varepsilon-1},
\end{equation*}
where we have used $\log(1+x)<x$.  We finally prove $\frac{1}{L}\int_{I_j}\sigma\, dx\lesssim \frac{R^a}{\mu}$, for each $j$, which concludes the proof since
\begin{equation*}
    \langle w\rangle_I\langle \sigma\rangle_I^{p-1}=\langle w\rangle_I\biggl(\frac{1}{L}\int_{I_1}\sigma\,dx +\frac{1}{L}\int_{I_2}\sigma\,dx \biggr)^{p-1}\lesssim R^{\varepsilon-1}R^{a(p-1)}\mu^{1-p}= m.
\end{equation*}
The case $j=2$ is straightforward since $x\leq R+L$ and $\sigma=x^a$ outside $M_k$ implies 
\begin{equation*}
    \frac{1}{L}\int_{I_2}\sigma\,dx \leq \frac{|I_2|}{L}(R+L)^a\lesssim R^a.
\end{equation*}
For $j=1$, we consider two cases. If $L\geq \frac{\mu r_k}{4}$, then
\begin{equation*}
    \frac{1}{L}\int_{I_1}\sigma\,dx \leq \frac{1}{L}\int_{M_k}\sigma\,dx \leq \frac{2}{L}r^{a+1}_k\leq \frac{8}{\mu} r^a_k\lesssim \frac{R^a}{\mu}.
\end{equation*}
If $L\leq \frac{\mu r_k}{4}$, then $I_1$ is contained in a corner of $M_k$, without touching the center. If $|I_1|=0$, the result is immediate; otherwise, by the symmetry of $\sigma$ around $r_k$, we can assume $I_1$ is in the right corner and for each $x\in I_1$, we have $x\geq (1+\mu )r_k-\frac{\mu r_k}{4}=(1+\frac{3 \mu}{4})r_k$, and since $\mu<1$, it follows that 
\begin{equation*}
    \sigma(x)=r^a_k\biggl(\frac{x-r_k}{\mu r_k}\biggr)^{\mu-1}\leq \frac{4}{3}r_k^a.
\end{equation*}
Thus,
\begin{equation*}
\begin{split}
    \frac{1}{L}\int_{I_1}\sigma&\leq \frac{1}{|I_1|}\int_{I_1} \sigma\leq \frac{4}{3}r^a_k \lesssim R^a .
    \qedhere
\end{split}
\end{equation*}
\end{proof}

\subsubsection{Construction of the symbol}
We next define a fixed $\text{BMO}(\mathbb{R})$ symbol $b$ corresponding to the family of lacunary weights $w_{p,m}$ previously constructed. 
\begin{lemma}
\label{lemLB4}
    If $B$ is defined by
\begin{equation*}
B(x):=\sum^\infty_{k=1}\log^+\biggl(\frac{r_{k+3}}{|x-r_k|}\biggr),
\end{equation*}
then $B\in \text{BMO}(\mathbb{R})$, and so is $b$, defined as
\begin{equation*}
    b(x):=\log^+ \frac{1}{|x+3|}+B(x).
\end{equation*}
\end{lemma}
\begin{proof}
    The function $\psi(t):=\log^+\frac{1}{|t|}=(-\log|t|)_+$ belongs to $\text{BMO}(\mathbb{R})$. For $k\in\mathbb{N}$, let
\begin{equation*}
    \psi_{k}(x):=\psi\biggl(\frac{x-r_k}{r_{k+3}}\biggr).
\end{equation*}
See that $B=\sum^\infty_{k=1}\psi_{k}$ and $\psi_{k}$ vanishes almost everywhere off $S_{k}:=(r_k-r_{k+3},r_k+r_{k+3})$. The $\text{BMO}(\mathbb{R})$ norm is preserved under translations and dilations, so $\lVert \psi_{k}\rVert_{\textrm{BMO}(\mathbb{R})}=\lVert \psi\rVert_{\textrm{BMO}(\mathbb{R})}$. \looseness=-1

Let $I$ be an interval. If $I$ intersects none of the sets $S_{k}$, then $B=0$ on $I$. If there is a unique $k$ such that $I$ intersects $S_{k}$, then $B=\psi_{k}$ on $I$ and $\langle |B-\langle B\rangle_I|\rangle_I\leq \lVert \psi\rVert_{\textrm{BMO}(\mathbb{R})}$. Suppose $I$ meets two or more supports. If $I\cap S_{k}\neq \varnothing$, then $r_{k+3}=\frac{1}{8}r_k$ implies 
\begin{equation}
\label{eq:Igeq}
    |I|\geq \frac{7}{8}r_k-\frac{9}{8}\frac{r_k}{2}>\frac{1}{4}r_k.
\end{equation}
Let $K$ be the smallest integer such that $I$ intersects $S_{K}$. See that 
\begin{equation}
\label{eq:Intpsipk}
    \int_\mathbb{R}\psi_{k} \,dx = 2 r_{k+3}\int^1_0 \log\frac{1}{t}\, dt=2 r_{k+3}\int^\infty_0 \lambda e^{-\lambda}\, d\lambda=r_{k+2}.
\end{equation}
Therefore,
\begin{equation*}
    \frac{1}{|I|}\int_I B\, dx\leq  \frac{1}{|I|}\sum_{k\geq K} \int_\mathbb{R} \psi_{k}\,dx = \frac{r_{K+1}}{|I|}\leq 2.
\end{equation*}
The last inequality follows from \eqref{eq:Igeq}. Since $B\geq 0$, the triangle inequality yields 
\begin{equation*}
    \frac{1}{|I|}\int_I|B-\langle B\rangle_I|\,dx \leq \frac{2}{|I|}\int_I B\, dx\leq 4.
\end{equation*}
Thus $B\in \text{BMO}(\mathbb{R})$. Finally, $b \in \text{BMO}(\mathbb{R})$ as it is the sum of two $\text{BMO}(\mathbb{R})$ functions.
\end{proof}

\begin{prop}
\label{prop4.5o}
If $p\in(1,\infty)$ and $m\geq N_p$ is an integer, then
\begin{equation*}
    \|[b,H]\|_{L^p(w_{p,m})\to L^{p,\infty}(w_{p,m})}\gtrsim [w_{p,m}]_{A_p}^{p'},
\end{equation*}
where $w_{p,m}$ is as defined in \eqref{eq:wpmDef} and $N_p$ is the smallest natural number such that $N_p^{-\frac{1}{p-1}}<\frac{1}{8}$. If $\alpha\in(0,1)$ and $v(x)=|x+3|^{\alpha-1}$, then
\begin{equation*}
    \|[b,H]\|_{L^p(v)\to L^{p,\infty}(v)}\gtrsim [v]_{A_p}^2.
\end{equation*}
\end{prop}
\begin{proof}
Let $w=w_{p,m}$ and define
\begin{equation*}
    G_k:=\biggl[\frac{5}{8}r_k, \frac{7}{8}r_k\biggr] \quad\text{and}\quad J_k:= (r_k-r_{k+3},r_k)
\end{equation*}
for $k\in \mathbb{N}$. Let 
\begin{equation*}
    E:=\bigcup^{2m+2}_{k=m+3}G_k,\quad \text{and} \quad F:=\bigcup^m_{k=1}J_k.
\end{equation*}
Note that if $x\in E$ and $y\in F$, then $x<y$ and $b(x)=0$. Therefore, for compactly supported $f\in L^r(\mathbb{R})$ with  $r\in(1,\infty)$, $b f\in L^1(\mathbb{R})$, and since $H$ is a CZO, we have 
\begin{equation*}
    [b,H] f(x)= \Bigl(b H(f)-H(b f)\Bigr)(x) = \int_{\mathbb{R}}-\frac{b(y)}{\pi (x-y)} f(y)\, dy
\end{equation*}
for almost every $x \in E$. Since $0<y-x<r_k$ for $x\in E$ and $y\in J_k$, the above equation implies that the kernel $K$ of $[b,H]$ satisfies 
\begin{equation*}
    K(x,y)\geq \chi_E(x) \biggl(\frac{b(y)}{\pi(y-x)}\chi_F(y)\biggr)\geq \chi_E(x)\biggl(\sum^m_{k=1} \frac{b(y)}{\pi r_k}\chi_{J_k}(y)\biggr)
\end{equation*}
for all $x\in E$ and $y\in F$.

Letting $q=p'$, by Theorem~\ref{thm:SepVar}, we have
\begin{equation}
\label{eq:4.6}
    \lVert [b,H]\rVert_{L^p(w)\to L^{p,\infty}(w)}\geq \biggl\lVert \sum^m_{k=1} \frac{b}{\pi r_k}\chi_{J_k}\biggr\rVert_{L^{q}(\sigma)}\lVert \chi_E\rVert_{L^{p,\infty}(w)}.
\end{equation}
Note $w(x)=x^{1/m-1}$ on $E$, and so
\begin{equation*}
\begin{split}
    \lVert \chi_E\rVert^p_{L^{p,\infty}(w)}&= \frac{(\frac{7}{8})^{\frac{1}{m}}-(\frac{5}{8})^\frac{1}{m}}{1/m}\sum^{2m+2}_{k=m+3}r_k^{\frac{1}{m}}\gtrsim \sum^{2m+2}_{k=m+3}1=m.
\end{split}
\end{equation*}
Since $J_1,\ldots,J_m$ are pairwise disjoint,
\begin{equation*}
     \biggl\lVert \sum^m_{k=1} \frac{b}{\pi r_k}\chi_{J_k}\biggr\rVert^q_{L^{q}(\sigma)}= \sum^m_{k=1} \biggl\lVert \frac{b}{\pi r_k}\chi_{J_k}\biggr\rVert^q_{L^{q}(\sigma)}.
\end{equation*}
Denote the $k^{\text{th}}$ summand on the right-hand side by $I(k)$.
The following equations describe how $\sigma$ and $b$ behave on $J^0_k:=((1-\mu_p)r_k,r_k)\subseteq J_k$: 
\begin{equation*}
\begin{split}
    &\sigma(y)=r_k^{(q-1)(1-\frac{1}{m})}\biggl(\frac{|y-r_k|}{\mu_p r_k}\biggr)^{\mu_p-1}, \quad \quad  y\in J_k^0,\\
    &b(y)=\log^+\biggl(\frac{r_k/8}{r_k-y}\biggr), \quad \quad y\in J^0_k.
\end{split}
\end{equation*}
If $a=(q-1)(1-\frac{1}{m})$, then
\begin{equation*}
\begin{split}
    I(k)&:=\int_{J_k} \biggl(\frac{b(\lambda)}{\pi r_k}\biggr)^{q}\sigma(\lambda)\, d\lambda\geq \int_{J_k^0} \biggl(\frac{b(\lambda)}{\pi r_k}\biggr)^{q}\sigma(\lambda)\, d\lambda\\
    &=\pi^{-q} r^{a-q}_k \int_{(1-\mu_p)r_k}^{r_k} \biggl[\log^+\biggl(\frac{r_k/8}{r_k-y}\biggr)\biggr]^{q} \biggl(\frac{r_k-y}{\mu_p r_k}\biggr)^{\mu_p-1}\ dy\\
    &\geq \pi^{-q} r^{a-q}_k \int_{(1-\mu_p)r_k}^{r_k} \biggl[\log\biggl(\frac{\mu_p r_k}{r_k-y}\biggr)\biggr]^{q} \biggl(\frac{r_k-y}{\mu_p r_k}\biggr)^{\mu_p-1}\ dy.
\end{split}
\end{equation*}
The last inequality comes from the fact that $\mu_p<\frac{1}{8}$.
Employing the monotone convergence theorem and two substitutions ($\xi=\frac{\mu_p r_k}{r_k-y}$ and $e^t=\xi^{\mu_p}$), we get
\begin{equation*}
\begin{split}
    I(k)&\geq \pi^{-q} r_k^{a+1-q}\mu_p^{-q}\Gamma(q+1).
\end{split}
\end{equation*}
By definition, $a+1-q=\frac{1}{m}(1-q)$, thus $\sum^m_{k=1} r_k^{a+1-q}>m2^{(1-q)}$. Therefore,
\begin{equation*}
\begin{split}
    \biggl\lVert \sum^m_{k=1} \frac{b}{\pi r_k}\chi_{J_k}\biggr\rVert_{L^{q}(\sigma)}=\biggl(\sum^m_{k=1}I(k)\biggr)^\frac{1}{q}\gtrsim \mu_p^{-1}m^{\frac{1}{q}}.
\end{split}
\end{equation*}

In conclusion, by \eqref{eq:4.6}, 
\begin{equation*}
\lVert [b,H]\rVert_{L^p(w)\to L^{p,\infty}(w)}\gtrsim \bigl(\mu_p^{-1}m^\frac{1}{q}\bigr) m^\frac{1}{p}=m^{q}\approx [w]^{p'}_{A_p}.
\end{equation*}

Now, let $w(x)=|x|^{\alpha-1}$ and $v(x)=w(x+3)$, with $\alpha\in(0,1)$. Denote $\sigma = w^{1-p'}$ and $\rho=v^{1-p'}$. See that $b$ equals $-\log|\cdot +3|$ on $X:=(-4,-3)$, and $b=0$ on $Y:=(-\infty,-4)$, so
\begin{equation*}
    K(x,y)=\frac{b(x)-b(y)}{\pi(x-y)}= \frac{-\log |x+3|}{\pi(x-y)}\geq \frac{-\log|x+3|}{\pi|y+3|}
\end{equation*}
for all $x\in X$ and $y\in Y$. 
Then, by Theorem~\ref{thm:SepVar},
\begin{equation*}
    \|[b,H]\|_{L^p(v)\to L^{p,\infty}(v)}\geq  \Bigl\lVert \frac{1}{\pi |\cdot+3|}\chi_{Y}\Bigr\rVert_{L^{p^\prime}(\rho)}\lVert -\log|\cdot+3|\chi_{X}\rVert_{L^{p,\infty}(v)}.
\end{equation*}
Under the reflection $x\mapsto -x-3$, the right-hand side becomes
\begin{equation*}
\Bigl\lVert \frac{1}{\pi |\cdot|}\chi_{(1,\infty)}\Bigr\rVert_{L^{p^\prime}(\sigma)}\lVert \log|\cdot|\chi_{(0,1)}\rVert_{L^{p,\infty}(w)}.
\end{equation*}
By Proposition~\ref{prop:thm1.1.1}, this expression is comparable to $[w]_{A_p}^2= [v]_{A_p}^2$. This finishes the proof.
\end{proof}

\subsubsection{Proof of the lower bound} 
We now establish the general lower bound, Theorem \ref{thm:GeneralLower}.

\begin{proof}[Proof of Theorem \ref{thm:GeneralLower}] Let $w_\alpha=|x+3|^{\alpha-1}$ for each $\alpha\in(0,1)$. Take $M>0$ so that \mbox{$[w_\alpha]_{A_p}\leq   M\frac{1}{\alpha}$}. For $t>M$, fix $\alpha=\frac{M}{t}$, and see that $[w_\alpha]_{A_p}\approx \frac{1}{\alpha} \approx t$. Then by Proposition~\ref{prop4.5o}
\begin{equation*}
    \varphi_{b,H}(t)\geq \lVert [b,H]\rVert_{L^p(w_\alpha)\to L^{p,\infty}(w_\alpha)}\gtrsim t^{2}.
\end{equation*}
Let $M>0$ be such that $ [w_{p,m}]_{A_p}\leq M m$ for every $m\geq N_p$. Let $t> MN_p$, $m= \lfloor t/M\rfloor$, and $w=w_{p,m}$. See that $[w]_{A_p}\approx m\approx t$, and in particular $[w]_{A_p}\leq t$. Then by Proposition~\ref{prop4.5o},
\begin{equation*}
    \varphi_{b,H}(t)\geq \|[b,H]\|_{L^p(w)\rightarrow L^{p,\infty}(w)} \gtrsim  t^{p^\prime}.
\end{equation*}
The reverse inequality follows from \eqref{eq:GeneralUpper}. 
\end{proof}

\section*{Artificial Intelligence Statement} 

The authors used ChatGPT (GPT-5.6 Sol and GPT-6 Astra) for literature review, which led to ideas inspiring the bound in Proposition \ref{prop:CommutatorHardyBound} and a construction leading to Theorem \ref{thm:GeneralLower}. After ChatGPT audited a preliminary version of this paper, the LLM suggested an updated symbol construction for Theorem \ref{thm:GeneralLower}, which is significantly simpler than our earlier example; this updated example is the impetus for our proof in Section \ref{section:GeneralLoer}. Additional AI use was limited to language editing and proofreading. The authors independently verified, validated, and rewrote all parts of the paper influenced by AI-generated material and take full responsibility for the paper's mathematical content.


\end{document}